\documentclass[12pt]{article}
\usepackage[margin=2cm]{geometry}
\usepackage{amsmath,amstext,amsfonts,amsthm,bbm}

\theoremstyle{plain}
\newtheorem{thm}{Theorem}
\newtheorem{defin}{Definition}[section]

\newtheorem{cor}[defin]{Corollary}

\newtheorem*{thm*}{Theorem}

\theoremstyle{remark}
\newtheorem{rem}[defin]{Remark}
\def\ch{\operatorname{ch}}

\begin{document}
\title{On  a uniqueness theorem of E.\ B.\ Vul}
\author{Sasha Sodin\footnote{School of Mathematical Sciences, Queen Mary University of London, London E1 4NS, United Kingdom \& School of Mathematical Sciences, Tel Aviv University, Tel Aviv, 69978, Israel.  Email: a.sodin@qmul.ac.uk. This work is supported in part by the European Research Council starting grant 639305 (SPECTRUM) and by a Royal Society Wolfson Research Merit Award.}}
\maketitle

\begin{abstract}
We recall a uniqueness theorem of E.\ B.\ Vul pertaining to a version of the cosine transform originating in spectral theory. Then we point out an application to the Bernstein approximation problem with non-symmetric weights: a theorem of Volberg is proved by elementary means. \end{abstract}

\section{Introduction}
\paragraph{A cosine transform and a uniqueness theorem}
The goal of this note is to draw attention to a uniqueness theorem for an integral transform originating in the spectral theory of Sturm--Liouville operators, and to point out an application in approximation theory. 

Let $\sigma: \mathbb R \to \mathbb R$ be a function of bounded variation such that
\begin{equation}
\forall x \geq 0 \,\, M_\sigma(x) \overset{\text{def}}= \int_{-\infty}^0 \exp(x \sqrt{|\lambda|}) \, | d\sigma(\lambda)| < \infty~.
\end{equation}
Define 
\begin{equation} (\mathfrak C\sigma)(x) =  \int_{-\infty}^\infty \cos(x \sqrt{\lambda}) d\sigma(\lambda)~, \quad x \in \mathbb R~,
\end{equation}
where (for $\lambda < 0$) $\cos(x \sqrt \lambda) = \ch (x \sqrt{-\lambda})$. The transform $\mathfrak C$ (and its variants) appeared in the works of Povzner \cite{Pov} and Krein \cite{Krein1}, who showed that a continuous even function $f: [-2a, 2a] \to \mathbb R$ defines an Hermitian-positive kernel $K_f(x, y) = f(x - y) + f(x+y)$ on $L_2(0, a)$ if and only if it has a representation $f = \mathfrak C \sigma$ with an increasing $\sigma$. 

These works led to the question whether the representation $f = \mathfrak C \sigma$ is unique for an increasing $\sigma$, i.e.\ for which (not necessarily increasing) $\sigma$ the equality $\mathfrak C \sigma \equiv 0$ implies that $\sigma \equiv \operatorname{const}$. Levitan proved \cite{Lev} that if $\mathfrak C \sigma \equiv 0$ and 
\[ \forall x \geq 0 \,\,  M_\sigma(x) \leq C \exp(C x^\alpha) \]
for some $\alpha < 2$, then $\sigma \equiv \operatorname{const}$. Levitan and Meiman \cite{LM} showed that the same conclusion remains valid for $\alpha = 2$. Finally, Vul proved \cite{Vul} the following definitive result:
\begin{thm}[Vul]\label{thm:vul} Suppose $p: \mathbb R_+ \to \mathbb R_+$ is a non-decreasing convex function.
\begin{enumerate}
\item  If  $p$ satisfies
\begin{equation}\label{eq:vulcond}
\int^\infty \frac{p(s)}{s^3} ds = \infty~,
\end{equation}
and $\sigma$ is a function of bounded variation such that 
\begin{equation}\label{eq:maj}
\forall x \geq 0\,\, M_\sigma(x) \leq C \exp(p^*(x)) \overset{\text{def}}{=} C \exp(\sup_{s \geq 0} \left[ x s - p(s) \right] )~.
\end{equation}
and $\mathfrak C \sigma \equiv 0$, then $\sigma \equiv \operatorname{const}$. 
\item If (\ref{eq:vulcond}) fails, there exists $\sigma$ of bounded variation such that  (\ref{eq:maj}) holds and $\mathfrak C \sigma \equiv 0$, and yet $\sigma \not\equiv \operatorname{const}$.
\end{enumerate}
\end{thm}
\begin{rem}In \cite{Vul}, this result is stated under an additional assumption 
\[ \lim_{x \to \infty} \frac{x {p^*}'(x)}{p^*(x)} = \gamma > 1~,\]
however, this requirement can be omitted with no essential modifications in the proof (which we reproduce in Section~\ref{s:vul}).\end{rem}

\bigskip\noindent
The uniqueness theorems of Levitan, Meiman, and Vul have found numerous applications. Already Krein \cite{Krein} used the result of \cite{LM} to provide sufficient conditions for the self-adjointness of one-dimensional second order differential operators in terms of the tails of (some) spectral measure. In \cite{LevSears}, Levitan showed that Theorem~\ref{thm:vul} implies the following sufficient condition, due to Sears \cite{Sears} for the self-adjointness of the Schr\"odinger operator $L = - \Delta + q$ on $\mathbb R^d$:
\begin{equation}\label{eq:sears} q(x) \geq - Q(|x|)~, \quad \text{where $Q > 0$ is monotone increasing and } \int^\infty \frac{dr}{\sqrt{Q(r)}} = \infty~.  \end{equation}
The earlier result of \cite{LM} recovers the weaker sufficient condition $q(x) \geq - Ax^2 - B$ due to Titchmarsh \cite{Titch}. Uniqueness results related to Theorem~\ref{thm:vul} have been also used in other branches of spectral theory, for example, in the study of the spectral edges of random band matrices \cite{me}. 

In this note, we present an application to the Bernstein approximation problem with non-symmetric weights, and also use the opportunity to make the proof of Vul accessible in English.

\paragraph{Bernstein's approximation problem} The setting is as follows: given a lower semicontinuous function $W: \mathbb R \to [1, \infty]$ such that $1/W(\lambda) = O(\lambda^{-\infty})$, one inquires whether polynomials are dense in the space 
\[ C_0(1/W) = \left\{ u \in C(\mathbb R)\, \mid \, \lim_{\lambda \to \infty} u(\lambda) / W(\lambda) = 0 \right\}~, \quad \|u\|_{C_0(1/W)} = \sup_\lambda |u(\lambda)|/W(\lambda)~. \]
Equivalently, consider the space
 \[ \mathfrak Q_{W} = \left\{ f(x) = \int e^{i x \lambda} d\sigma(\lambda) \, : \, \int W(\lambda) |d\sigma(\lambda)| < \infty \right\}~. \]
By a  Hahn--Banach argument, polynomials are dense in $C_0(1/W)$ if and only if $\mathfrak Q_{W}$ is quasianalytic in the sense of Hadamard, i.e.\ a function $f \in \mathfrak Q_{W}$ that has a zero of infinite order has to vanish identically. 
 
\medskip
We recall some of the classical results on the Bernstein approximation problem, and refer to the surveys \cite{Akh} and \cite{Merg} and the book \cite{Koosis} for more details. Hall  \cite{Hall} and Izumi and Kawata \cite{IK} proved that the condition 
\begin{equation}\label{eq:hall}
\int_{-\infty}^\infty \frac{\log W(\lambda)}{1 + \lambda^2} d\lambda = + \infty 
\end{equation}
is necessary for the completeness of polynomials in $C_0(1/W)$. Carleson showed \cite{Car} that in general, (\ref{eq:hall}) is not sufficient. However \cite{IK,Car}, if a weight $W(\lambda)$ satisfying (\ref{eq:hall}) is even, and the function $s \mapsto \log W(e^s)$ is convex, then polynomials are dense in $C_0(1/W)$, and consequently also in $C_0(1/\tilde W)$ for any $\tilde W \geq W$. Still, there are weights $\tilde W$ for which polynomials are dense and yet no such minorant $W \leq \tilde W$ exists. 
Several necessary and sufficient conditions for completeness were derived, particularly, by Akhiezer--Bernstein, Pollard, Mergelyan (see \cite{Akh, Merg}), de Branges \cite{DeB}, and more recently by Poltoratski \cite{Polt}, yet these are not always easily verifiable. 

Now we turn to the case when $W$ has different rates of growth at $\pm \infty$.  In the extreme case $W(\lambda) \equiv +\infty$ for $\lambda < 0$, it follows from the result of Hall and Izumi--Kawata (for example, using the map $f \to \phi_f$ in (\ref{eq:phi}) below) that the condition  
\begin{equation}\label{eq:hall'}
\int_0^\infty \frac{\log W(\lambda)}{1 + \lambda^{3/2}} d\lambda = + \infty 
\end{equation}
is necessary for the completeness of polynomials in $C_0(1/W)$. Consequently, if $W$ is an arbitrary weight, it is necessary for the completeness of polynomials in $C_0(1/W)$ that the pair of conditions 
\begin{equation}
\int_0^\infty \frac{\log W(\lambda)}{1 + \lambda^{3/2}} d\lambda = + \infty~, \quad 
 \int_0^\infty \frac{\log W(-\lambda)}{1 + \lambda^{2}} d\lambda = + \infty
\end{equation}
be satisfied by either $W(\lambda)$ or $W(-\lambda)$. In \cite{Vol}, Volberg showed, answering a question asked by Mergelyan  and  by Ehrenpreis (cf.\ \cite[Problem 13.8]{Ehr}), that the conditions 
\begin{align} \label{eq:vola}
&\int^\infty \frac{\log W(\lambda)}{\lambda^{3/2}} d\lambda = + \infty
\\
&\int^\infty \frac{\log W(-\lambda)}{\lambda^{2}} d\lambda = + \infty~
\label{eq:volb}
\end{align} are  sufficient for completeness, provided that $W$ is regular in the following sense:
\begin{equation}\label{eq:reg1}
\log W(\lambda) = \begin{cases}
\sqrt{\lambda} \epsilon_1(1/\lambda)~, &\lambda \geq 0 \\
|\lambda|\epsilon_2(1/|\lambda|)~, & \lambda< 0
\end{cases}~, \quad
\lim_{t \to +0} \epsilon_i(t) = \lim_{t \to +0} \frac{t\epsilon_i'(t)}{\epsilon_i(t)} = \lim_{t \to +0} \frac{t^2\epsilon_i''(t)}{\epsilon_i(t)} = 0~.
\end{equation}
Another proof of this result was given by Borichev \cite{Bor}.\footnote{The results of \cite{Vol, Bor} are stated for $L_2(1/W)$, whereas here we focus on the space $C_0(1/W)$. The difference is not essential for the current discussion, as both the original results and the current Theorem~\ref{thm:vol} are valid in both cases; cf.\ Corollary~\ref{cor:l2}. We also mention the work of Bakan \cite{Bak,Bak1}, who found a general connection between completeness in $C_0(1/W)$ and $L_2(1/W)$.}   The proof in \cite{Vol} is based on a construction of an auxiliary analytic function and delicate estimates of the harmonic measure, whereas that of \cite{Bor} is based on the method of quasianalytic (or almost holomorphic) extension, put forth by Dyn$'$kin  \cite{Dyn}.

The result of \cite{Vol} was significantly generalised by M.~Sodin \cite{Sod}, who relied on a theorem of de Branges \cite{DeB}. It is shown in \cite{Sod} that if 
\begin{align}\tag{\ref{eq:vola}$'$}
&\text{polynomials are normally dense in $C_0(\mathbbm{1}_{[0, \infty)}/W)$} \\
\tag{\ref{eq:volb}}
& \int^\infty \frac{\log W(-\lambda)}{\lambda^{2}} dx = + \infty~
\end{align}
and $W$ satisfies the regularity assumption
\begin{equation}\tag{\ref{eq:reg1}$'$}
\log W(\lambda) = |\lambda| \epsilon_2(1/|\lambda|) \text{ for $\lambda < 0$}, \quad \text{and} \quad \epsilon_2(t) \searrow 0~, \frac{t\epsilon_2'(t)}{\epsilon_2(t)} \searrow 0\quad \text{as $t \to +0$}
\end{equation}
then polynomials are also dense in $C_0(1/W)$. Here (\ref{eq:vola}$'$) means that polynomials are dense for any weight differing from $W/\mathbbm{1}_{[0, \infty)}$ at a finite number of points; this condition is necessary, and thus can not be further relaxed. The regularity condition (\ref{eq:reg1}$'$) is somewhat restrictive, however, it is shown in \cite{Sod} that some regularity has to be imposed: there exists a weight $W$ satisfying (\ref{eq:vola}) and (\ref{eq:volb}) such that the functions $s \mapsto \log W(\pm e^s)$ are convex and still polynomials are not dense in $C_0(1/W)$. 

\medskip Here we prove
\begin{thm}\label{thm:vol}
Let $W: \mathbb R \to [1, \infty]$ be a (lower semicontinuous) function such that $\frac{1}{W(\lambda)}= O(\lambda^{-\infty})$, and (\ref{eq:vola}) and (\ref{eq:volb}) hold. If the functions $s \mapsto \log W(e^s)$ and $s \mapsto \log W(-s^2)$ are convex on $[s_0, \infty)$, then polynomials are dense in $C_0(1/W)$.
\end{thm}
\noindent Equivalently, the class $\mathfrak Q_W$ is quasianalytic; {\em a fortiori}, a non-zero $f \in \mathfrak Q_W$ can not vanish on a set of positive measure. 

\medskip\noindent
Our condition on $W|_{\mathbb R_+}$ is much more stringent than the  optimal condition (\ref{eq:vola}$'$) of \cite{Sod} (although less stringent than (\ref{eq:reg1})). On the other hand, the regularity assumptions on $W|_{\mathbb R_-}$ are weaker than (\ref{eq:reg1}$'$).

More importantly, the proof of Theorem~\ref{thm:vol} is relatively elementary. It is based on the following well-known construction, similar to the one used to relate the Stieltjes moment problem to the Hamburger one. Let 
\[ f(x) = \int_{-\infty}^\infty e^{i x \lambda} d\sigma(\lambda)~,\]
where $\sigma: \mathbb R \to \mathbb R$ is a function of bounded variation such that
\[ \forall k \geq 0 \quad \int_0^\infty  |\lambda|^k  |d\sigma(\lambda)| < \infty~, \quad \forall x \geq 0 \quad  \int_{-\infty}^0 e^{x \sqrt{|\lambda|}} d\sigma(\lambda) < \infty~.\]
Denote 
\begin{equation}\label{eq:phi}\phi_f(x) =  (\mathfrak C \sigma)(x) = \int_{-\infty}^\infty \cos(x \sqrt \lambda) d\sigma(\lambda)~.\end{equation}
Observing that 
\[ f^{(k)}(0) = i^k \int \lambda^k d\sigma(\lambda)~, \quad 
\phi_f^{(2k)}(0) = (-1)^k  \int \lambda^k d\sigma(\lambda)~, \quad 
\phi_f^{(2k+1)}(0) = 0~,\]
we see that if $f$ has a zero of infinite order at $x = 0$, then so does $\phi_f$.  

In the recent  note \cite{me2}, we presented an application of the map $f \mapsto \phi_f$ to a problem of analytic quasianalyticity, which corresponds to the case when $\sigma$ is supported on $\mathbb R_+$. Here we use this map  and the Denjoy--Carleman theorem to reduce Theorem~\ref{thm:vol} to Theorem~\ref{thm:vul}. We note that both the Denjoy--Carleman theorem and Theorem~\ref{thm:vul} can be proved (see \cite[\S 14.3]{Levin} and Section~\ref{s:vul} below, respectively) using but the Carleman theorem from complex analysis, and the latter is a direct consequence of the formula for the harmonic measure in the half-plane.

\paragraph{Several corollaries} The following corollaries are derived from Theorem~\ref{thm:vol} by relatively standard methods. The proofs are sketched in Section~\ref{s:cor}.

\begin{cor}\label{cor:l2}  Let  $W$ be as in Theorem~\ref{thm:vol}. For any  measure $\mu \geq 0$ with  $\int W(\lambda) d\mu(\lambda) < \infty$, polynomials are dense in $L_2(\mu)$.
\end{cor}

\begin{cor}\label{cor:moment}  Let  $W$ be as in Theorem~\ref{thm:vol}. Any measure  $\mu \geq 0$ with $\int W(\lambda) d\mu(\lambda) < \infty$  is Hamburger determinate, i.e.\ it shares its moments with no other non-negative measure on $\mathbb  R$.
\end{cor}

\smallskip
Finally, we deduce a variant of a result of Volberg from \cite{Vol2}.
\begin{cor}\label{cor:fr} Let $W$ be as in Theorem~\ref{thm:vol}, and let $\nu \geq 0$ be a measure on $\mathbb R$ such that 
\[ \int W(1/\lambda) d\nu(\lambda) < \infty~.\]
Let $(z_j = a_j + i b_j)_{j = 1}^\infty$ be a sequence of points in the upper half-plane such that $z_j \to 0$ and
\begin{equation}\label{eq:nontang}
\lim_{j \to \infty} \frac{\log W(1/a_j)}{\log (1/b_j)} = \infty~.
\end{equation}
Then the linear span of $\left\{ \frac{1}{t - z_j} \right\}_{j \geq 1}$ is dense in $L_2(\nu)$.
 \end{cor}
 Here the regularity assumptions on $W(\lambda)$ are somewhat weaker than in \cite{Vol2} (where a condition of the form (\ref{eq:reg1}$'$) and its counterpart at $+\infty$ are imposed), and the assumption of non-tangential convergence $|b_j| \geq \epsilon |a_j|$ is relaxed to (\ref{eq:nontang}).

\section{Proof of Theorem~\ref{thm:vol}}
Assume that $f \in \mathfrak Q_W$ has a zero of infinite order, say, at $x = 0$: $f^{(k)}(0) = 0$ for all $k \geq 0$. Then $\phi_f$ of (\ref{eq:phi}) also satisfies $\phi_f^{(k)}(0) = 0$ for all $k \geq 0$, and
\begin{equation}\label{eq:estderphi} \sup_{|y| \leq x} |\phi_f^{(k)}(y)| \leq C M_k(x) \overset{\text{def}}{=} C \max_{\lambda \geq 0} \lambda^{k/2} \left\{
W(-\lambda)^{-1} e^{x \sqrt{\lambda}} + W(\lambda)^{-1} \right\}~. \end{equation}
Let us show that 
\begin{equation}\label{eq:quas} \forall x \geq 0  \quad \sum_{k \geq 1} M_k(x)^{-1/k} = \infty~. \end{equation}
To this end, set 
\[ B_k  = \max_{\lambda \geq 0} \lambda^{k/2} W(\lambda)^{-1} = \exp(\max_{s} [\frac{k}{2} s - \log W(e^s)]) = 
\exp(q^*(k/2))~,  \]
where $q(s) = \log W(e^s)$. From the convexity of $q$, we deduce (following Ostrowski \cite{Ostr}) that the condition (\ref{eq:vola}) implies (and is equivalent to)
\begin{equation}\label{eq:quas_bk} \sum_{k=1}^\infty B_k^{-1/k} = \infty~. \end{equation}
Now, from (\ref{eq:volb}) we have that $\log W(-\lambda) \geq (x + 1) \sqrt{\lambda}$ for sufficiently large $\lambda \geq \lambda_0(x)$, therefore 
\[ \max_{\lambda \geq 0} \lambda^{k/2} W(-\lambda)^{-1} e^{x \sqrt\lambda} 
\leq \lambda_0^k e^{x \sqrt{\lambda_0}} 
   + \max_{\lambda \geq 0} \lambda^{k/2} e^{- \sqrt\lambda}  
 \leq C_x^{k+1} k!\]
Thus
\[ \sum_{k=1}^\infty M_k(x)^{-1/k} \geq c_x \sum_{k=1}^\infty \min(B_k^{-1/k}, 1/ k) = \infty~,\]
where on the last step we used (\ref{eq:quas_bk}) and the Cauchy condensation test. 
According to the Denjoy--Carleman theorem \cite{Carl}, \cite[\S 14.3]{Levin}, the class of functions admitting an estimate (\ref{eq:estderphi}) with $M_k$ satisfying (\ref{eq:quas}) is quasianalytic in the sense of Hadamard, whence $\phi_f \equiv 0$.

Now we appeal to Theorem~\ref{thm:vul}. Let $p(s)$ be the largest convex minorant of $p_0(s) = \log W(- s^2)$, $s \geq 0$. The functions $p$ and $p_0$ coincide for large $s$, therefore
\[ \int^\infty \frac{p(s)}{s^3} ds = \infty~.\]
We have:
\[ \int_{-\infty}^0 \exp(x \sqrt{|\lambda|}) |d\sigma(\lambda)| 
\leq  \sup_{\lambda < 0} W(\lambda)^{-1}  e^{x \sqrt{|\lambda|}} \int_{-\infty}^0 W(\lambda) |d\sigma(\lambda)|
\leq C   \exp(p^*(x))~.\]
Adjusting the constants, we can assume that $p$ is non-decreasing. Therefore  Theorem~\ref{thm:vul} applies and we obtain $\sigma \equiv \operatorname{const}$ and $f \equiv 0$. \qed

\section{Proof of Corollaries}\label{s:cor}

\begin{proof}[Proof of Corollary~\ref{cor:l2}]
Observe that $W_1(\lambda) = \sqrt{W(\lambda)}$ also satisfies the conditions of Theorem~\ref{thm:vol}, and that for any $u \in C_0(1/W_1)$
\[ \|u\|_{L_2(\mu)}^2 = \int |u(\lambda)|^2 d\mu(\lambda)
 \leq \left[ \sup |u(\lambda)| / W_1(\lambda) \right]^2 \, \int W(\lambda) d\mu(\lambda) = C  \| u \|_{C_0(1/W_1)}^2~.\]
Any function $L_2(\mu)$ can be approximated by functions in $C_0(1/W)$, and, by Theorem~\ref{thm:vol}, these in turn can be approximated by polynomials.
\end{proof}
We mention again that a general reduction of the problem of completeness in $L_2(1/W)$ to that in $C_0(1/W)$ was found by Bakan in \cite{Bak,Bak1}.

\begin{proof}[Proof of Corollary~\ref{cor:moment}]
Pick a measure $\mu_1 \geq \mu$ be a measure which is not discrete and still satisfies $\int W(\lambda) d\mu_1(\lambda) < \infty$ (for example, one can add to $\mu$ a small continuous component near the origin). By corollary~\ref{cor:l2}, polynomials are dense in $L_2(\mu_1)$. By a theorem of M.\ Riesz \cite[Theorem 2.3.3]{AkhBook}, $\mu_1$ is $N$-extreme. Further \cite[\S 3.4.1]{AkhBook}, an $N$-extreme measure is either moment-determinate or discrete (or both). Our $\mu_1$ is not discrete, hence it is moment-determinate, and hence so is $\mu$.
\end{proof}

\begin{proof}[Proof of Corollary~\ref{cor:fr}]
The condition (\ref{eq:nontang}) implies that
\begin{equation}\label{eq:nontang-cor} \forall k \geq 1 \quad c_k = \inf_j  \left[  W(1/a_j) b_j^k \right] > 0~. \end{equation}
Consider the domain 
\begin{equation} \Omega = \left\{  z = a + i b \, : \,  -1 \leq  a \leq 1~, \,\, b > 0 ~,\,\, \forall k \geq 1 \quad W(1/a) b^k \geq c_k\right\}~.\end{equation}
Then 
\begin{equation}\label{eq:unif} \forall k \geq 1 \quad \inf_{z \in \Omega} 
\inf_{t \in \mathbb{R}} \, [W(1/t)|t - z|^k] \geq 
\inf_{z = a+ib \in \Omega} \, [W(1/a) b^k] = c_k~. \end{equation}
 If $E = \operatorname{span} \left\{ \frac{1}{t - z_j} \right\}_{j \geq 1}$ is not dense in $L_2(\nu)$, let $u(t) \in E^\perp \setminus \{0\}$, and set 
\[ g(z) = \int \frac{\bar u(t) d\nu(t)}{t - z}~. \]
Then $g(z_j) = 0$ for $j=1,2,\cdots$. Further, (\ref{eq:unif}) implies that $g$ and its derivatives are uniformly bounded in $\Omega$:
\[\begin{split} |g^{(k)}(z)| 
&\leq k!  \int \frac{|u(t)| d\nu(t)}{|t-z|^{k+1}} 
\leq k! \, \|u \|_{L_2(\nu)} \, \sqrt{ \int \frac{d\nu(t)}{|t-z|^{2k+2}}} \\
&\leq k! \, \|u \|_{L_2(\nu)} \, \sqrt{ \int  W(1/t) d\nu(t) \times \sup_{t} \frac{1}{W(1/t) |t-z|^{2k+2}}}\\
&\leq \frac{k!}{\sqrt{c_{2k+2}}} \, \|u \|_{L_2(\nu)} \, \sqrt{ \int  W(1/t) d\nu(t)} ~, \end{split}\]
hence $g$ has a zero of infinite order at $z = 0$, i.e.\ 
\[ \forall k \geq 1 \, \quad \, \int \lambda^{k}  \bar u(1/\lambda) d\nu(1/\lambda) = \int \frac{\bar u(t) d\nu(t)}{t^k} = 0~. \]
Let $W_1(\lambda) = \max(1, \sqrt{W(\lambda)} / (1 + |\lambda|))$. Then $W_1$ satisfies the assumptions of Theorem~\ref{thm:vol}, and the complex measure $d\sigma(\lambda) = \lambda \bar u(1/\lambda) d\nu(1/\lambda)$
satisfies
\[ \int W_1(\lambda) |d\sigma(\lambda)|
= \int W_1(\lambda) |\lambda \bar u(1/\lambda)| d\nu(1/\lambda)
\leq \|u \|_{L_2(\nu)} \sqrt{\int W_1(\lambda)^2 \lambda^2 d\nu(1/\lambda)} 
< \infty~. \]
Therefore the measures $(\Re \sigma)_\pm$ in the Jordan decomposition of $\Re \sigma$ satisfy the assumptions of Corollary~\ref{cor:moment} and share the same moments, which implies that $(\Re \sigma)_+ = (\Re \sigma)_-$; similarly,$(\Im \sigma)_+ = (\Im \sigma)_-$, whence $u \equiv 0$ (in $L_2(\nu)$).
\end{proof}

\section{Proof of Theorem~\ref{thm:vul}}\label{s:vul}

We closely follow \cite{Vul}. Assume that (\ref{eq:vulcond}) and (\ref{eq:maj}) hold and that $\mathfrak C \sigma \equiv 0$. Adjusting constants, we may assume that $p(0) = 0$; in this case the $s \mapsto p(s)/ s$ is non-decreasing: 
\[ \frac{d}{ds} \frac{p(s)}{s} = - \frac{p(s) - s p'(s)}{s^2} \geq - \frac{p(0)}{s^2} = 0\]
(at the differentiability points of $p$).

From (\ref{eq:maj}) we have for $\lambda_0 < 0$ and $x \geq 0$:
\begin{equation}\label{eq:dec1}
\int_{-\infty}^{\lambda_0} |d\sigma(\lambda)| \leq \exp(- x \sqrt{|\lambda_0|}) 
\times C \exp(p^*(x))~,
\end{equation}
whence
\begin{equation}\label{eq:dec}
\int_{-\infty}^{\lambda_0} |d\sigma(\lambda)| \leq C \exp(-p(\sqrt{|\lambda_0|}))~.
\end{equation}
Consider the Stieltjes transform 
\[ F(z) = \int \frac{d\sigma(\lambda)}{\lambda - (z +i)}~, \quad \Im z \geq 0~. \]
We have: $|F(z)| \leq C$. If we show that 
\begin{equation}\label{eq:car} \int_{- \infty}^0 \frac{\log |F(x)|}{1 + x^2} dx = - \infty~,\end{equation}
Carleman's theorem \cite[\S 14.2]{Levin} will imply that $F \equiv 0$ and hence $\sigma \equiv \operatorname{const}$. Therefore we turn to the proof of (\ref{eq:car}).

For $x < 0$ let 
\[ F(x) = F_1(x) + F_2(x)~, \quad F_1(x) = \int_{-\infty}^{x/4} \frac{d\sigma(\lambda)}{\lambda - (x +i)}~, \quad  F_2(x) = \int_{x/4}^\infty \frac{d\sigma(\lambda)}{\lambda - (x +i)}~.\]
From (\ref{eq:dec})
\begin{equation}\label{eq:est1} |F_1(x)| \leq \int_{-\infty}^{x/4} |d\sigma(\lambda)| \leq C \exp(-p(\sqrt{|x|}/2))~.\end{equation}
To prove an estimate of the similar form for $F_2(x)$, we start from the identity
\begin{equation}\label{eq:id} \frac{1}{\lambda - (x+i)} = \frac{i}{\sqrt{x+i}} \int_0^\infty \cos(u \sqrt \lambda) \exp(i u \sqrt{x + i}) du~,\end{equation}
valid for $x < 0$, $\Im \sqrt \lambda < \Im \sqrt{x + i}$. From (\ref{eq:id}) we have (for $x < x_0 < 0$):
\[\begin{split} F_2 (x) &= \frac{i}{\sqrt{x +i}} \int_{x/4}^\infty d\sigma(\lambda) \int_0^\infty \cos(u \sqrt \lambda) \exp(i u \sqrt{x + i}) du \\
&= \frac{i}{\sqrt{x+i}} \int_{x/4}^\infty d\sigma(\lambda) \left[ \int_0^{u^*}  + \int_{u^*}^\infty \right]
= I_1(x) + I_2(x)~,\end{split}\]
where we take $u^* = \frac{p(\sqrt{|x|}/2)}{2 \sqrt{|x|}}$. From the assumption $\mathfrak C \sigma \equiv 0$,
\[I_1 (x) = \frac{i}{\sqrt{x+i}} \int_{-\infty}^{x/4} d\sigma(\lambda) \int_0^{u^*} \cos(u \sqrt \lambda) \exp(i u \sqrt{x + i}) du~;\]
estimating 
\[ \left| \frac{i}{\sqrt{x+i}} \int_0^{u^*} \cos(u \sqrt \lambda) \exp(i u \sqrt{x + i}) du \right|
\leq \int_0^{u^*} \cos(u \sqrt{\lambda}) du \leq \exp(u^* \sqrt{|\lambda|})~,\]
we obtain using (\ref{eq:dec}) and (\ref{eq:maj}):
\begin{equation}\label{eq:est2.1-0}
\begin{split}
|I_1(x)| &\leq \int_{-\infty}^{x/4}\exp(u^* \sqrt{|\lambda|}) |d\sigma(\lambda)| \\
&\leq \sqrt{\int_{-\infty}^{x/4} |d\sigma(\lambda)|} \sqrt{\int_{-\infty}^{x/4}\exp(2 u^* \sqrt{|\lambda|}) |d\sigma(\lambda)|} \\
&\leq C \exp(-\frac{1}{2} p(\sqrt{|x|}/2) + \frac{1}{2} p^*(2u^*))~.\end{split}\end{equation}
Let us show that $p^*(2u^*) \leq \frac12 p(\sqrt{|x|}/2)$. Let $s^*$ be such that 
\[ 2u^* s^* = p^*(2u^*) + p(s^*)~, \]
then 
\[ p(s^*) / s^* \leq 2 u^* \leq 4 u^* = p(\sqrt{|x|}/2)/(\sqrt{|x|}/2)~,\]
whence $s^* \leq \sqrt{|x|}/2$ and
\[ p^*(2u^*) \leq 2u^* s^* \leq u^*\sqrt{|x|} = \frac12 p(\sqrt{|x|}/2)~, \]
as claimed. Returning to (\ref{eq:est2.1-0}), we obtain that
\begin{equation}\label{eq:est2.1}
|I_1(x)| \leq C \exp(-\frac14 p(\sqrt{|x|}/2))~. \end{equation}
Now we turn to $I_2(x)$. Using that
\[\left| \frac{i}{\sqrt{x+i}} \int_{u^*}^\infty \cos(u \sqrt \lambda) \exp(i u \sqrt{x + i}) du \right|
\leq \int_{u^*}^\infty  \exp(- u \sqrt{|x|}/2)  du \leq \exp(- u^* \sqrt{|\lambda|}/2)~, \]
we obtain:
\begin{equation}\label{eq:est2.2}
|I_2(x)| \leq C \exp(- u^* \sqrt{|x|}/2 ) =  C \exp(-\frac14 p(\sqrt{|x|}/2))~.
\end{equation}
Combining (\ref{eq:est1}), (\ref{eq:est2.1}), and (\ref{eq:est2.2}) we obtain that
\[ |F(x)| \leq 3 C \exp(-\frac14 p(\sqrt{|x|}/2))~,\]
whence
\[ \int_{-\infty}^{x_0} \frac{\log |F(x)|}{1 + x^2} dx \leq - C_1 - c_1 \int_{|x_0|}^\infty \frac{p(\sqrt{|x|}/2)}{x^2} dx \leq - C_1 - c_2 \int_{\sqrt{|x_0|}}^\infty \frac{p(s)}{s^3} ds = - \infty~,\]
as claimed.

\medskip Vice versa, suppose that (\ref{eq:vulcond}) fails. Then there exists a non-zero entire function $\Phi(z)$ such that 
\[ \forall z = x + iy \in \mathbb C: \quad |\Phi(z)| \leq  \exp(|y| - p(\sqrt{|z|}) - \sqrt{|z|})\]
(see \cite[Lemma 5]{Mand}, where such a function is constructed as a product of dilated cardinal sine functions). Then $\Phi_1(z) =  e^{iz} \Phi(z)$ satisfies 
\[ \forall z \in \mathbb C_+: \quad |\Phi_1(z)| \leq  \exp(- p(\sqrt{|z|}) - \sqrt{|z|})~.\]
Let 
\[ \sigma(\lambda) = \int_{-\infty}^\lambda \Re \Phi_1(\lambda') d\lambda' \quad \text{or} \quad  \sigma(\lambda) = \int_{-\infty}^\lambda \Im \Phi_1(\lambda') d\lambda'~, \]
so that $\sigma \not\equiv \operatorname{const}$. Shifting the integration contour up, we see that $\mathfrak C \sigma \equiv 0$.  We also have:
\[ \int_{-\infty}^0 \exp(x \sqrt{|\lambda|}) |d\sigma(\lambda)| \leq
\int_{0}^\infty \exp(x \sqrt{\lambda} - p(\sqrt{\lambda}) - \sqrt{\lambda}) d\lambda \leq C \exp (p^*(x))~.\qed\]

\paragraph{Acknowledgement.} I am grateful to A.\ Kiro, M.\ Sodin and A.\ Volberg for helpful comments.

\end{document}